\documentclass{louloupart}
\usepackage{pinlabel}
\usepackage{graphicx}
\usepackage[all]{xy}
\title[Word problem and parabolic subgroups in Dyer groups]{Word problem and parabolic subgroups in Dyer groups}
\author[L Paris]{Luis Paris}
\givenname{Luis}
\surname{Paris}
\address{Luis Paris, IMB, UMR 5584, CNRS, Universit\'e Bourgogne Franche-Comt\'e, 21000 Dijon, France}
\email{lparis@u-bourgogne.fr}
\author[M Soergel]{Mireille Soergel}
\givenname{Mireille}
\surname{Soergel}
\address{Mireille Soergel, ETH Z\"urich, Mathematics Department, CH-8092 Z\"urich, Switzerland}
\email{mireille.soergel@math.ethz.ch}


\newtheorem{thm}{Theorem}[section]
\newtheorem{lem}[thm]{Lemma}
\newtheorem{prop}[thm]{Proposition}
\newtheorem{corl}[thm]{Corollary}

\theoremstyle{definition}
\newtheorem*{defin}{Definition}
\newtheorem*{rem}{Remark}
\newtheorem{expl}[thm]{Example}
\newtheorem*{acknow}{Acknowledgments}

\numberwithin{equation}{section}

\makeatletter
\renewcommand{\thefigure}{\ifnum \c@section>\z@ \thesection.\fi
 \@arabic\c@figure}
\@addtoreset{figure}{section}
\makeatother

\begin{document}

\def\N{\mathbb N} \def\Z{\mathbb Z} \def\DD{\mathcal D}
\def\id{{\rm id}} \def\Glex{{\rm glex}} \def\Supp{{\rm Supp}}
\def\Pc{{\rm Pc}} \def\PP{\mathcal P}


\begin{abstract}
One can observe that Coxeter groups and right-angled Artin groups share the same solution to the word problem.
On the other hand, in his study of reflection subgroups of Coxeter groups Dyer introduces a family of groups, which we call Dyer groups, which contains both, Coxeter groups and right-angled Artin groups.
We show that all Dyer groups have this solution to the word problem, we show that a group which admits such a solution belongs to a little more general family  of groups that we call quasi-Dyer groups, and we show that this inclusion is strict.
Then we show several results on parabolic subgroups in quasi-Dyer groups and in Dyer groups.
Notably, we prove that any intersection of parabolic subgroups in a Dyer group of finite type is a parabolic subgroup.

\smallskip\noindent
{\bf AMS Subject Classification\ \ } 
Primary: 20F36, 20F55.
Secondary: 20F10.

\smallskip\noindent
{\bf Keywords\ \ } 
Coxeter groups, right-angled Artin groups, Dyer groups, word problem, parabolic subgroups.

\end{abstract}

\maketitle


\section{Introduction}\label{sec1}

There is an extensive literature on Coxeter groups as well as on right-angled Artin groups and, more generally, on graph products of cyclic groups.
A peculiarity of these two families of groups is that they share the same solution to the word problem: that given by Tits \cite{Tits1} for Coxeter groups and that given by Green \cite{Green1} for graph products of cyclic groups.
This common algorithm goes beyond the simple solution to the word problem since it provides an effective criterion to determine if an expression is reduced or not and it makes it possible to define normal forms.

It is therefore natural to ask the following questions.
What do these two families of groups have in common that makes them to have the same solution to the word problem?
Which other groups have this solution to the word problem?

To be more precise, our questions and study concern marked groups.
Recall that a \emph{marked group} is a pair $(G,X)$ where $G$ is a group and $X$ is a generating set for $G$.
Note that Coxeter groups and graph products of cyclic groups are actually marked groups.
In the rest of the paper we will say that a marked group $(G,X)$ has \emph{Property $\DD$} if it has the same solution to the word problem as Coxeter groups and as graph products of cyclic groups.
A precise definition of Property $\DD$ is given in Section \ref{sec2}.

In his study of reflection subgroups of Coxeter groups Dyer \cite{Dyer1} introduces a family of groups, that we call \emph{Dyer groups}, which contains both, Coxeter groups and graph products of cyclic groups.
A careful reading of the proof of \cite[Lemma 2.8]{Dyer1} completed with ideas from Tits \cite{Tits1} allows an informed reader to show that these groups have Property $\DD$.
A complete and explicit proof of this result is given in Section \ref{sec3} (see Theorem \ref{thm2_2}).

This partially answers our questions in the sense that it says what Coxeter groups and graph products of cyclic groups have in common that makes them to have the same solution to the word problem.
It remains to see whether Dyer groups are the only marked groups to have Property $\DD$, and, if not, which other groups have it.

Examples of marked groups that are not Dyer groups and that have Property $\DD$ can be easily found.
For example $(\Z,\{1,2\})$ has Property $\DD$ and it is not a Dyer group (note that $(\Z,\{1\})$ is a Dyer group).
However, in this context it is reasonable to restrict the study to marked groups $(G,X)$ satisfying the following property.
\begin{itemize}
\item
For all $x,y\in X$ and $a,b\in\Z$ such that $x^a\neq1$ and $y^b\neq1$, if $x^ay^b$ is a non-trivial power of an element of $X$ or if $x^ay^b=1$, then $x=y$.
\end{itemize}
In that case we say that $(G,X)$ is a \emph{strongly marked group}.
Coxeter groups, graph products of cyclic groups and, more generally, Dyer groups are strongly marked groups.

We introduce a family of marked groups a little more general than that of Dyer groups which we call \emph{quasi-Dyer groups} and in Section \ref{sec4} we show that, if a strongly marked group $(G,X)$ has Property $\DD$, then $(G,X)$ is a quasi-Dyer group (see Theorem \ref{thm2_3}).

We do not know if all quasi-Dyer groups are strongly marked and/or if they all have Property $\DD$, but in Section \ref{sec5} we show a family of quasi-Dyer groups that are not Dyer groups, that are strongly marked, and that have Property $\DD$ (see Proposition \ref{prop2_4}).

The next question that motivates this work is: what properties common to Coxeter groups and right-angled Artin groups can be extended to Dyer groups and, more generally, to groups with Property $\DD$?
A first answer can be found in Soergel \cite{Soerg1} where actions of Dyer groups on CAT(0) spaces are constructed that extend those of Coxeter groups on Davis--Moussong complexes (see Moussong \cite{Mouso1}) and those of right-angled Artin groups on Salvetti complexes (see Charney--Davis \cite{ChaDav1}).

Parabolic subgroups play a prominent role in the study of Coxeter groups and in that of right-angled Artin groups.
Part of the results on these subgroups extends to groups with Property $\DD$.
In particular we show that a parabolic subgroup of a group having Property $\DD$ has Property $\DD$ (see Proposition \ref{prop2_7}) and that the intersection of two standard parabolic subgroups is a standard parabolic subgroup (see Corollary \ref{corl2_6}).
However, the uniqueness property for an element of minimal syllabic length in a coset of a standard parabolic subgroup holds for Dyer groups (see Proposition \ref{prop2_8}) but not for all groups having Property $\DD$ (see Example \ref{expl2_9}).

In Section \ref{sec6} we show that any intersection of parabolic subgroups in a Dyer group of finite type is a parabolic subgroup (see Theorem \ref{thm2_10}).
This property is known and widely used for Coxeter groups (see Tits \cite{Tits2}, Solomon \cite{Solom1}, Krammer \cite{Kramm1} and Qi \cite{Qi1}).
It is also known for right-angled Artin groups (see Duncan--Kazachkov--Remeslennikov \cite{DuKaRe1}) but, as far as we know, it is new for graph products of cyclic groups.

The paper is organized as follows.
In Section \ref{sec2} we give detailed and precise definitions and statements.
We also show that the intersection of two standard parabolic subgroups in a group with Property $\DD$ is a standard parabolic subgroup.
In Section \ref{sec3} we prove that Dyer groups have Property $\DD$.
In Section \ref{sec4} we show that any strongly marked group with Property $\DD$ is a quasi-Dyer group.
In Section \ref{sec5} we show a family of quasi-Dyer groups that are strongly marked, that have Property $\DD$, and that are not Dyer groups.
In Section \ref{sec6} we show that the intersection of a family of parabolic subgroups in a Dyer group of finite type is a parabolic subgroup.

\begin{acknow}
Both authors are supported by the French project ``AlMaRe'' (ANR-19-CE40-0001-01) of the ANR.
\end{acknow} 


\section{Definitions and statements}\label{sec2}

We first recall the definitions of a Coxeter group and of a graph product of cyclic groups.

\begin{defin}
Let $\Gamma$ be a simplicial graph with vertex set $V(\Gamma)$ and edge set $E(\Gamma)$.
We assume that $E(\Gamma)$ is endowed with a map $m:E(\Gamma)\to\N_{\ge 2}$.
To the pair $(\Gamma,m)$ we associate a group $W=W(\Gamma,m)$, called a \emph{Coxeter group}, defined by the following presentation:
\begin{gather*}
W=\langle x_v\,,\ v\in V(\Gamma)\mid x_v^2=1\text{ for all }v\in V(\Gamma)\,,\ (x_ux_v)^{m(e)}=1\\
\text{ for all }e=\{u,v\}\in E(\Gamma)\rangle\,.
\end{gather*}
The pair $(W,X)$ is called a \emph{Coxeter system}, where $X=\{x_v\mid v\in V(\Gamma)\}$.
\end{defin}

\begin{defin}
Let $\Gamma$ be a simplicial graph with vertex set $V(\Gamma)$ and edge set $E(\Gamma)$.
We assume that $V(\Gamma)$ is endowed with a map $f:V(\Gamma)\to\N_{\ge 2}\cup\{\infty\}$.
To the pair $(\Gamma,f)$ we associate a group $C=C(\Gamma,f)$, called a \emph{graph product of cyclic groups}, defined by the following presentation:
\begin{gather*}
C=\langle x_v\,,\ v\in V(\Gamma)\mid x_v^{f(v)}=1\text{ for all }v\in V(\Gamma)\text{ such that }f(v)\neq\infty\,,\ x_ux_v=x_vx_u\\
\text{ for all }e=\{u,v\}\in E(\Gamma)\rangle\,.
\end{gather*}
If $f(v)=\infty$ for all $v\in V(\Gamma)$, then $C$ is called a \emph{right-angled Artin group}.
\end{defin}

We turn now to recall the solution to the word problem common to these two families of groups.
Throughout the paper we use the following notations.
The order of an element $g$ in a group is denoted by $o(g)$.
If $o(g)$ is finite, then $\Z_{o(g)}=\Z/o(g)\Z$ is the cyclic group of order $o(g)$, and if $o(g)$ is infinite, then $\Z_{o(g)}=\Z$ is the infinite cyclic group.
Recall that a \emph{marked group} is a group $G$ endowed with a generating set $X\subset G$.
A marked group $(G,X)$ is of \emph{finite type} if $X$ is finite.
Let $(G,X)$ be a marked group.
\emph{The set of syllables} of $X$ is
\[
S(X)=\{x^a\mid x\in X\,,\ a\in\Z_{o(x)}\setminus\{0\}\}\,.
\]
It is clear that $S(X)$ also generates $G$.
Our solution to the word problem uses words on $S(X)$ and not on $X$, but in the concrete cases of Coxeter groups, of graph products of cyclic groups, and, more generally, of Dyer groups, this is not a problem.

Let $(G,X)$ be a marked group.
We denote by $S(X)^*$ the free monoid on $S(X)$.
The elements of $S(X)^*$ are called \emph{syllabic words} and they are written as finite sequences.
The concatenation of two words $w,w'\in S(X)^*$ is written $w\cdot w'$.
If $w =(s_1,s_2,\dots,s_\ell)\in S(X)^*$ is a syllabic word, then we set $\overline{w}=s_1s_2\cdots s_\ell\in G $ and we say that $\overline{w}$ is \emph{represented} by $w$.
The shortest length of a syllabic word representing an element $g\in G$ is called the \emph{syllabic length} of $g$ and it is denoted by $\lg_{S(X)}(g)$.
A syllabic word $w=(s_1,s_2,\dots,s_\ell)$ is \emph{reduced} if $\ell=\lg_{S(X)}(\overline{w})$.
If $a,b$ are two letters and $m$ is an integer $\ge 2$, then we denote by $[a,b]_m$ the alternating word $(a,b,a,\dots)$ of length $ m$.
Similarly, if $a,b$ are two elements of a group $G$, then we denote by $\overline{[a,b]}_m$ the alternating product $aba\cdots$ of length $m$.

\begin{defin}
Let $(G,X)$ be a marked group and let $w\in S(X)^*$ be a syllabic word.
Assume that $w$ can be written as $w=w_1\cdot(s,t)\cdot w_2$, where $w_1,w_2\in S(X)^*$, $s,t\in S (X)$, and $st\in S(X)\cup\{1\}$.
Set 
\[
w'=\left\{\begin{array}{ll}
w_1\cdot(st)\cdot w_2&\text{if }st\neq 1\,,\\
w_1\cdot w_2&\text{if }st=1\,.
\end{array}\right.
\]
Then we say that we can go from $w$ to $w'$ through an \emph{elementary M-operation of type I}.
Assume that $w$ can be written as $w=w_1\cdot [s,t]_m\cdot w_2$, where $w_1,w_2\in S(X)^*$, $s,t\in S(X)$, $m\ge2$, $\overline{[s,t]}_m=\overline{[t,s]}_m$ and $\lg_{S(X)}(\overline{[s,t]}_m)=m$.
Set
\[
w'=w_1\cdot[t,s]_m\cdot w_2\,.
\]
Then we say that we can go from $w$ to $w'$ through an \emph{elementary M-operation of type II}.
We say that $w$ is \emph{M-reduced} if its length cannot be shortened by any finite sequence of elementary M-operations.
\end{defin}

Notice that elementary M-operations of type I strictly decrease lengths of syllabic words while  elementary M-operations of type II preserve lengths.
Furthermore, elementary M-operations of type II are reversible but not those of type I.
Notice also that, if we can go from $w$ to $w'$ through a finite sequence of elementary M-operations, then $\overline{w}=\overline{w'}$.

\begin{defin}
A marked group $(G,X)$ is said to have \emph{Property $\DD$} if
\begin{itemize}
\item[(a)]
for all $w\in S(X)^*$, $w$ is reduced if and only if $w$ is M-reduced, and
\item[(b)]
for all $w,w'\in S(X)^*$, if $w$ and $w'$ are both reduced and $\overline{w}=\overline{w'}$, then we can go from $w$ to $w'$ through a finite sequence of elementary M-operations of type II.
\end{itemize}
\end{defin}

\begin{rem}
Property $\DD$ for a marked group $(G,X)$ of finite type solves the word problem in $(G,X)$, but the algorithm also solves the following two other questions.
\begin{itemize}
\item
Given a syllabic word $w$, the algorithm determines a reduced syllabic word $w'$ such that $\overline{w} = \overline{w'}$.
\item
Given a syllabic word $w$, the algorithm determines whether $w$ is reduced or not.
\end{itemize}
\end{rem}

The solutions to the word problem for Coxeter groups by Tits \cite{Tits1} and for graph products of cyclic groups by Green \cite{Green1} are summarized in the following theorem.

\begin{thm}[Tits \cite{Tits1}, Green \cite{Green1}]\label{thm2_1}
\begin{itemize}
\item[(1)]
Let $W=W(\Gamma,m)$ be a Coxeter group and let $X=\{x_v\mid v\in V(\Gamma)\}$ be its standard generating set.
Then $(W,X)$ has Property $\DD$.
\item[(2)]
Let $C=C(\Gamma,f)$ be a graph product of cyclic groups and let $X=\{x_v\mid v\in V(\Gamma)\}$ be its standard generating set.
Then $(C,X)$ has Property $\DD$.
\end{itemize}
\end{thm}

\begin{defin}
Let $\Gamma$ be a simplicial graph with vertex set $V(\Gamma)$ and edge set $E(\Gamma)$.
We assume that $E(\Gamma)$ is endowed with a map $m:E(\Gamma)\to\N_{\ge 2}$ and that $V(\Gamma)$ is endowed with a map $f:V(\Gamma)\to\N_{\ge 2}\cup\{\infty\}$.
We further assume that, for each $e=\{u,v\}\in E(\Gamma)$, if $m(e)\neq2$, then $f(u)=f(v)=2$.
To the triple $(\Gamma,m,f)$ we associate a group $D=D(\Gamma,m,f)$, called a \emph{Dyer group}, defined by the following presentation:
\begin{gather*}
D=\langle x_v\,,\ v\in V(\Gamma)\mid x_v^{f(v)}=1\text{ for all }v\in V(\Gamma)\text{ such that }f(v)\neq\infty\,,\\
[x_u,x_v]_{m(e)}=[x_v,x_u]_{m(e)}\text{ for all }e=\{u,v\}\in E(\Gamma)\rangle\,.
\end{gather*}
\end{defin}

Observe that any Coxeter group is a Dyer group and that any graph product of cyclic groups is a Dyer group.
Moreover, it is shown in Dyer \cite{Dyer1} that the map $V(\Gamma)\to D$, $v\mapsto x_v$, is injective, hence $(D,X)$ is a marked group, where $X=\{x_v\mid v\in V(\Gamma)\}$.
This marked group is called a \emph{Dyer system}.

As pointed out in the introduction, an informed reader familiar with Coxeter groups will implicitly find the proof of the following theorem in Dyer \cite{Dyer1}.
However, the groups studied in the present paper include other groups such as right-angled  Artin groups, and therefore the paper is not addressed only to experts in Coxeter groups.
So, we give an explicit and complete proof of the following theorem in Section \ref{sec3}.

\begin{thm}\label{thm2_2}
Every Dyer system has Property $\DD$.
\end{thm}

As mentioned in the introduction, it is quite easy to find marked groups that have Property $\DD$ and that are not Dyer systems.
For example $(\Z,\{1,2\})$ is a marked group which has Property $\DD$ but which is not a Dyer system.
So, to make the study more coherent we impose the following additional hypothesis on marked groups.

\begin{defin}
Let $(G,X)$ be a marked group.
We say that $(G,X)$ is \emph{strongly marked} if $1\not\in X$ and it satisfies the following condition:
\begin{itemize}
\item
Let $x,y\in X$, $s\in\langle x\rangle\setminus\{1\}$ and $t\in\langle y\rangle\setminus\{1\}$.
If $st\in S(X)\cup\{1\}$, then $x=y$.
\end{itemize}
In particular, if $x,y\in X$, $x\neq y$, then $\langle x\rangle\cap\langle y\rangle=\{1\}$.
\end{defin}

Let $W=W(\Gamma,m)$ be a Coxeter group and let $X=\{x_v\mid v\in V(\Gamma)\}$ be its standard generating set.
Then $(W,X)$ is strongly marked.
Similarly, if $C=C(\Gamma,f)$ is a graph product of cyclic groups and $X=\{x_v\mid v\in V(\Gamma)\}$ is its standard generating set, then $(C,X)$ is strongly marked.
More generally, if $D=D(\Gamma,m,f)$ is a Dyer group and $X=\{x_v\mid v\in V(\Gamma)\}$, then $(D,X) $ is strongly marked (see Dyer \cite{Dyer1}).

For our next theorem we need to slightly extend the notion of Dyer group as follows.

\begin{defin}
Let $\Gamma$ be a simplicial graph with vertex set $V(\Gamma)$ and edge set $E(\Gamma)$.
We assume that $E(\Gamma)$ is endowed with a map $m:E(\Gamma)\to\N_{\ge 2}$ and that $V(\Gamma)$ is endowed with a map $f:V(\Gamma)\to\N_{\ge 2}\cup\{\infty\}$.
We further assume that, for each $e=\{u,v\}\in E(\Gamma)$,
\begin{itemize}
\item
if $m(e)>2$ and $m(e)$ is even, then $f(u)=f(v)=2$,
\item
if $m(e)>2$ and $m(e)$ is odd, then $f(u)$ and $f(v)$ are both finite and even and at least one of them is equal to $2$.
\end{itemize}
To the triple $(\Gamma,m,f)$ we associate a group $QD=QD(\Gamma,m,f)$, called a \emph{quasi-Dyer group}, defined by the following presentation:
\begin{gather*}
QD=\langle x_v\,,\ v\in V(\Gamma)\mid x_v^{f(v)}=1\text{ for all }v\in V(\Gamma)\text{ such that }f(v)\neq\infty\,,\ x_ux_v=x_vx_u\\
\text{ for all }e=\{u,v\}\in E(\Gamma)\text{ such that }m(e)=2\,,\ [x_u^{f(u)/2},x_v^{f(v)/2}]_{m(e)}=[x_v^{f(v)/2},x_u^{f(u)/2}]_{m(e)}\\
\text{ for all }e=\{u,v\}\in E(\Gamma)\text{ such that }m(e)>2\rangle\,.
\end{gather*}
The pair $(QD,X)$ is called a \emph{quasi-Dyer system}.
\end{defin}

\begin{rem}
In the above definition, if we replace the condition ``if $m(e)>2$ and $m(e)$ is odd, then $f(u)$ and $f(v)$ are both even and at least one of them is equal to $2$'' by the condition ``if $m(e)>2$ and $m(e)$ is odd, then $f(u)$ and $f(v)$ are both equal to $2$'' then $QD(\Gamma,m,f)=D(\Gamma,m,f)$ is a Dyer group.
\end{rem}

The main result of Section \ref{sec4} is the following.

\begin{thm}\label{thm2_3}
Let $(G,X)$ be a strongly marked group.
If $(G,X)$ has Property $\DD$, then $(G,X)$ is a quasi-Dyer system.
\end{thm}

Let $(G,X)$ be a strongly marked group having Property $\DD$.
Let $\Gamma$ be a simplicial graph and let $m:E(\Gamma)\to\N_{\ge 2}$ and $f:V(\Gamma)\to\N_{\ge 2}\cup\{\infty\}$ be maps with the right conditions so that $(\Gamma,m,f)$ defines a quasi-Dyer group with $G=QD(\Gamma,m,f)$ and $X=\{x_v \mid v\in V(\Gamma)\}$.
Then we say that $QD(\Gamma,m,f)$ is a \emph{quasi-Dyer presentation} for $(G,X)$.

We do not know if the reciprocal of Theorem \ref{thm2_3} is true, that is, if all quasi-Dyer systems have Property $\DD$.
We do not know either if quasi-Dyer systems are all strongly marked groups. 
However, we know a family of quasi-Dyer systems that are not Dyer systems, that are strongly marked, and that have Property $\DD$.
The following result is proved in Section \ref{sec5}.

\begin{prop}\label{prop2_4}
Let $m\ge3$ odd and $k\ge2$.
Let 
\[
QD_{m,k}=\langle x,y\mid x^2=y^{2k}=1\,,\ [x,y^k]_m=[y^k,x]_m\rangle\,.
\]
Then $(QD_{m,k},\{x,y\})$ is strongly marked and has Property $\DD$.
\end{prop}

\begin{defin}
If $(G,X)$ is a marked group and $Y\subset X$, then we denote by $G_Y$ the subgroup of $G$ generated by $Y$ and we say that $(G_Y,Y)$ is a \emph{standard (marked) parabolic subgroup} of $(G,X)$.
If $Y\subset X$ and $g\in G$, then $(gG_Yg^{-1},gYg^{-1})$ is simply called a \emph{(marked) parabolic subgroup}.
\end{defin}

Let $(G,X)$ be a strongly marked group having Property $\DD$.
The \emph{support} of a syllabic word $w=(x_1^{a_1},x_2^{a_2},\dots,x_\ell^{a_\ell})$ is $\Supp(w)=\{x_1,x_2,\dots,x_\ell\}$.
This is well-defined since $(G,X)$ is strongly marked.
Let $g\in G$.
We choose a reduced syllabic form $w=(x_1^{a_1},x_2^{a_2},\dots,x_\ell^{a_\ell})$ for $g$ and we define the \emph{support} of $g$ as $\Supp(g)=\Supp(w)=\{x_1,\dots,x_\ell\}\subset X$.
Since one can go from a reduced syllabic form of $g$ to another through a finite sequence of elementary M-operations of type II and elementary M-operations of type II do not change supports of syllabic words, the definition of $\Supp(g)$ does not depend on the choice of the reduced syllabic form.
On the other hand, it is easily seen that, if $w$ and $w'$ are two syllabic  words such that one can go from $w$ to $w'$ through a finite sequence of elementary M-operations, then $\Supp(w)\supset\Supp(w')$.
This proves the following.

\begin{lem}\label{lem2_5}
Let $(G,X)$ be a strongly marked group with Property $\DD$.
Let $Y\subset X$ and $g\in G$.
We have $g\in G_Y$ if and only if $\Supp(g)\subset Y$.
\end{lem}

A direct consequence of this lemma is the following.

\begin{corl}\label{corl2_6}
Let $(G,X)$ be a strongly marked group with Property $\DD$.
For $Y,Y'\subset X$ we have $G_Y\cap G_{Y'}=G_{Y\cap Y'}$.
\end{corl}

Let $(G,X)$ be a strongly marked group with Property $\DD$ and let $QD(\Gamma,m,f)$ be a quasi-Dyer presentation for $(G,X)$.
For $U\subset V(\Gamma)$ we denote by $\Gamma_U$ the full subgraph of $\Gamma$ spanned by $U$, we denote by $m_U:E(\Gamma_U)\to\N_{\ge 2}$ the restriction of $m$ to $E(\Gamma_U)$, and we denote by $f_U:V(\Gamma_U)\to\N_{\ge2}\cup\{\infty\}$ the restriction of $f$ to $V(\Gamma_U)$.
The first result which is proved in Section \ref{sec6} is the following.

\begin{prop}\label{prop2_7}
Let $(G,X)$ be a strongly marked group with Property $\DD$ and let $QD(\Gamma,m,f)$ be a quasi-Dyer presentation for $(G,X)$.
Let $U\subset V(\Gamma)$ and $Y=\{x_u\mid u\in U\}$.
Then $(G_Y,Y)$ is a strongly marked group with Property $\DD$ and $QD(\Gamma_U,m_U,f_U)$ is a quasi-Dyer presentation for $(G_Y,Y)$.
\end{prop}

The second result of Section \ref{sec6} concerns only Dyer systems and not strongly marked groups with Property $\DD$.
Indeed, as shown in Example \ref{expl2_9}, this result does not hold for all strongly marked groups with Property $\DD$.

\begin{prop}\label{prop2_8}
Let $(D,X)$ be a Dyer system, $Y\subset X$, and $g\in D$.
\begin{itemize}
\item[(1)]
There exists a unique element $g_0$ in $gD_Y$ of minimal syllabic length, and this element satisfies $\lg_{S(X)}(g_0h)=\lg_{S(X)}(g_0)+\lg_{S(X)}(h)$ for all $h\in D_Y$.
\item[(2)]
There exists a unique element $g_0$ in $D_Yg$ of minimal syllabic length, and this element satisfies $\lg_{S(X)}(hg_0)=\lg_{S(X)}(h)+\lg_{S(X)}(g_0)$ for all $h\in D_Y$.
\end{itemize}
\end{prop}

\begin{expl}\label{expl2_9}
Let $G=QD_{3,2}=\langle x,y\mid x^2=y^4=1\,,\ xy^2x=y^2xy^2\rangle$ and $X=\{x,y\}$.
We know by Proposition \ref{prop2_4} that $(G,X)$ is strongly marked and has Property $\DD$.
Let $Y=\{x\}$ and $g=yxy^2$.
Then $gG_Y$ has two elements, $g=yxy^2$ and $gx=y^3xy^2$, and $\lg_{S(X)}(g)=\lg_{S(X)}(gx)=3$.
In particular $gG_Y$ does not have a unique element of minimal syllabic length.
\end{expl}

The main result of Section \ref{sec6} is the following.

\begin{thm}\label{thm2_10}
Let $(D,X)$ be a Dyer system of finite type and let $\{P_i\mid i\in I\}$ be a non-empty collection of parabolic subgroups of $D$.
Then $\bigcap_{i\in I}P_i$ is a parabolic subgroup of $D$.
\end{thm}

We do not know if this result remains true if we remove the hypothesis ``to be of finite type'', however we can prove that the intersection of two parabolic subgroups is always a parabolic subgroup (see Lemma \ref{lem6_2}).

The following consequence of Theorem \ref{thm2_10} is widely used in the theory of Coxeter groups (see Kammer \cite{Kramm1}, for example).
On the other hand, an equivalent statement for Artin groups is one of the central questions in the field.

\begin{corl}\label{corl2_11}
Let $(D,X)$ be a Dyer system of finite type and let $A$ be a subset of $D$.
Then there exists a smallest (for the inclusion) parabolic subgroup containing $A$.
\end{corl}

Let $(D,X)$ be a Dyer system and let $A\subset D$ be a subset.
As for Coxeter groups and for Artin groups the smallest parabolic subgroup containing $A$ is denoted by $\Pc(A)$ and is called the \emph{parabolic closure} of $A$.


\section{Dyer systems}\label{sec3}

As mentioned in Section \ref{sec2} the aim of the present section is to prove Theorem \ref{thm2_2}.
We first recall some results on Dyer groups proved in Dyer \cite{Dyer1}.

Let $(D,X)$ be a Dyer system.
Let $R=\{gxg^{-1}\mid g\in D\,,\ x\in X\}$.
For each $\rho\in R$ we take a copy $H_\rho=\{a\,[\rho]\mid a\in\Z_{o(\rho)}\}$ of $\langle\rho\rangle$ whose operation is denoted additively.
We consider the abelian group
\[
M(D,X)=\bigoplus_{\rho\in R}H_\rho\,,
\]
that we endow with a structure of $D$-module, where the action of an element $g\in D$ on an element $m=\sum_{\rho\in R}a_\rho[\rho]$ is defined by
\[
g\cdot m=\sum_{\rho\in R}a_\rho[g\rho g^{-1}]\,.
\]

Let $g\in D$.
Choose a syllabic representative $w=(x_1^{a_1},x_2^{a_2},\dots,x_p^{a_p})$ for $g$.
For each $i\in\{1,\dots,p\}$ we set
\[
\rho_i=x_1^{a_1}\cdots x_{i-1}^{a_{i-1}}x_ix_{i-1}^{-a_{i-1}}\cdots x_1^{-a_1}\in R\,.
\]
Then we set
\[
N(g)=\sum_{i=1}^pa_i[\rho_i]\in M(D,X)\,.
\]
The following theorem gathers together some results proved in Dyer \cite{Dyer1}.

\begin{thm}[Dyer \cite{Dyer1}]\label{thm3_1}
Let $(D,X)$ be a Dyer system.
\begin{itemize}
\item[(1)]
Let $g\in D$.
Then the definition of $N(g)\in M(D,X)$ does not depend on the choice of the syllabic representative for $g$.
\item[(2)]
Let $g\in G$.
Let $N(g)=\sum_{\rho\in R}a_\rho(g)\,[\rho]$.
Then $\lg_{S(X)}(g)=|\{\rho\in R\mid a_\rho(g)\neq0\}|$.
\item[(3)]
Let $g,h\in D$.
Then $N(gh)=N(g)+g\cdot N(h)$.
\end{itemize}
\end{thm}

Now using Theorem \ref{thm3_1} we prove a version for Dyer groups of the so-called Exchange Lemma.

\begin{lem}\label{lem3_2}
Let $(D,X)$ be a Dyer system.
Let $g\in D$ and let $w=(x_1^{a_1},\dots,x_\ell^{a_\ell})$ be a reduced syllabic expression for $g$.
For all $i\in\{1,\dots,\ell\}$ we set
\[
\rho_i=x_1^{a_1}\cdots x_{i-1}^{a_{i-1}}x_ix_{i-1}^{-a_{i-1}}\cdots x_1^{-a_1}\in R\,.
\]
Let $s_0=x_0^{a_0}\in S(X)$.
If $\lg_{S(X)}(s_0g)\le\ell$, then there exists $i\in\{1,\dots,\ell\}$ such that $x_0=\rho_i$.
In that case, if $a_0+a_i= 0$, then $(x_1^{a_1},\dots,x_{i-1}^{a_{i-1}},x_{i+1}^{a_{i+1}},\dots,x_\ell^{a_\ell})$ is a reduced syllabic expression for $s_0g$, and if $a_0+a_i\neq 0$, then $(x_1^{a_1},\dots,x_{i-1}^{a_{i-1}},x_i^{a_0+a_i},x_{i+1}^{a_{i+1}},\dots,x_\ell^{a_\ell})$ is a reduced syllabic expression for $s_0g$.
\end{lem}

\begin{proof}
We have $N(g)=\sum_{i=1}^\ell a_i[\rho_i]$ and, by Theorem \ref{thm3_1}\,(2), $\rho_i\neq\rho_j$ for $i\neq j$.
By Theorem \ref{thm3_1}\,(3),
\[
N(s_0g)=a_0[x_0]+\sum_{i=1}^\ell a_i[s_0\rho_is_0^{-1}]\,.
\]
Since the $s_0\rho_is_0^{-1}$ are pairwise distinct for $i\in\{1,\dots,\ell\}$ and $\lg_{S(X)}(s_0g)\le\ell$, by Theorem \ref{thm3_1}\,(2) there exists $i\in\{1,\dots,\ell\}$ such that $s_0\rho_is_0^{-1}=x_0$, hence $\rho_i=s_0^{-1}x_0s_0=x_0$.

Suppose $a_0+a_i=0$.
From the equality $x_0=\rho_i$ follows that $s_0g=x_1^{a_1}\cdots x_{i-1}^{a_{i-1}}x_{i+1}^{a_{i+1}}\cdots x_\ell^{a_\ell}$, hence $w'=(x_1^{a_1},\dots,x_{i-1}^{a_{i-1}},x_{i+1}^{a_{i+1}},\dots,x_\ell^{a_\ell})$ is a syllabic expression for $s_0g$.
Furthermore,
\[
N(s_0g)=\sum_{j=1}^{i-1}a_j[s_0\rho_js_0^{-1}]+\sum_{j=i+1}^\ell a_j[s_0\rho_js_0^{- 1}]\,,
\]
and the $[s_0\rho_js_0^{-1}]$ are pairwise distinct for $j\in\{1,\dots,\ell\}$, hence, by Theorem \ref{thm3_1}\,(2), $\lg_{S(X)}(s_0g)=\ell-1$.
So, $w'$ is a reduced syllabic word.

Suppose $a_0+a_i\neq 0$.
From the equality $x_0=\rho_i$ follows that $s_0g=x_1^{a_1}\cdots x_{i-1}^{a_{i-1}}x_i^{a_0+a_i}x_{i+1} ^{a_{i+1}}\cdots x_\ell^{a_\ell}$, hence $w'=(x_1^{a_1},\dots,x_{i-1}^{a_{i-1}},x_i^{a_0+a_i},x_{i+1}^{a_{i+1}},\dots,x_\ell^{a_\ell})$ is a syllabic expression for $s_0g$.
Furthermore, 
\[
N(s_0g)=\sum_{j=1}^{i-1}a_j[s_0\rho_js_0^{-1}]+(a_0+a_i)[s_0\rho_is_0^{-1}]+\sum_{j=i+1}^\ell a_j
[s_0\rho_js_0^{-1}]\,,
\]
and the $[s_0\rho_js_0^{-1}]$ are pairwise distinct for $j\in\{1,\dots,\ell\}$, hence, by Theorem \ref{thm3_1}\,(2), $\lg_{S(X)}(s_0g)=\ell$.
So, $w'$ is a reduced syllabic word.
\end{proof}

\begin{proof}[Proof of Theorem \ref{thm2_2}]
Let $(D,X)$ be a Dyer system.
We start by showing that, if $w$ and $w'$ are two reduced syllabic words such that $\overline{w} = \overline{w'}$, then we can go from $w$ to $w'$ through a finite sequence of elementary M-operations of type II. 
We denote by $\ell$ the common length of $w$ and $w'$, and we argue by induction on $\ell$.
The cases $\ell=0$ and $\ell=1$ are trivial, hence we can assume that $\ell\ge2$ and that the induction hypothesis holds.

We set $w=(s_1,\dots,s_\ell)$ and $w'=(t_1,\dots,t_\ell)$.
If $s_1=t_1$, then $w_1=(s_2,\dots,s_\ell)$ and $w_1'=(t_2,\dots,t_\ell)$ are two reduced syllabic words such that $\overline{w_1}=\overline{w_1'}$.
Thus, by the induction hypothesis, we can go from $w_1$ to $w_1'$ through a finite sequence of elementary M-operations of type II.
It follows that we can go from $w$ to $w'$ through a finite sequence of elementary M-operations of type II.
So, we can assume that $s_1\neq t_1$.

We prove the following claim by induction on $q\ge1$.

{\it Claim.}
Let $q\ge1$.
If $\overline{[s_1,t_1]}_p\neq\overline{[t_1,s_1]}_p$ for all $p<q$, then $\lg_{S(X)}(\overline{[s_1,t_1]}_q)=\lg_{S(X)}(\overline{[t_1,s_1]}_q)=q$ and there exist reduced syllabic words $u_q$ and $u_q'$ of length $\ell-q$ such that $\overline{w}=\overline{w'}=\overline{[s_1,t_1]_q\cdot u_q}=\overline{[t_1,s_1]_q\cdot u_q'}$.

{\it Proof of the claim.}
The case $q=1$ is obtained directly by setting $u_q=(s_2,\dots,s_\ell)$ and $u_q'=(t_2,\dots,t_\ell)$.
So we can assume that $q\ge2$ and that the induction hypothesis on $q$ holds.
By the induction hypothesis $\lg_{S(X)}(\overline{[s_1,t_1]}_{q-1})=\lg_{S(X)}(\overline{[t_1,s_1]}_{q-1}) =q-1$ and there exist reduced syllabic words $u_{q-1}$ and $u_{q-1}'$ of length $\ell-q+1$ such that $\overline{w}=\overline{w'}=\overline{[s_1,t_1]_{q-1}\cdot u_{q-1}}=\overline{[t_1,s_1]_{q-1}\cdot u_{q-1}'}$.
We set $[t_1,s_1]_{q-1}\cdot u_{q-1}'=(r_1,\dots,r_\ell)$.
We have $r_i=t_1$ if $i$ is odd and $i\le q-1$, $r_i=s_1$ if $i$ is even and $i\le q-1$, and $u_{q-1}'=(r_q,r_{q+1},\dots,r_\ell)$.
Since $\lg_{S(X)}(s_1^{-1}\overline{w})<\ell$ and $(r_1,\dots,r_\ell)$ is a reduced syllabic expression for $\overline{w}$, by Lemma \ref{lem3_2} there exists $i\in\{1,\dots,\ell\}$ such that $(r_1,\dots,\widehat{r_i},\dots,r_\ell)$ is a reduced syllabic expression for $s_1^{-1}\overline{w}$.
Thus $(s_1,r_1,\dots,\widehat{r_i},\dots,r_\ell)$ is a reduced syllabic expression for $\overline{w}$.
If we had $i=1$ and $q\ge 3$, then we would have $r_2=s_1$, thus $(s_1,\widehat{r_1},r_2,\dots,r_\ell)=(s_1,s_1,r_3,\dots,r_\ell)$ would not be a reduced syllabic word: contradiction.
If we had $2\le i\le q-2$, then we would have $r_{i-1}=r_{i+1}$, hence $(r_1,\dots,\widehat{r_i},\dots,r_\ell)$ would not be reduced: contradiction.
If we had $i=q-1$ ($q=2$ and $i=1$ included), then $(s_1,r_1,\dots,\widehat{r_i},\dots,r_\ell)=[ s_1,t_1]_{q-1}\cdot u_{q-1}'$ would be a reduced syllabic expression for $\overline{w}=\overline{[t_1,s_1]_{q-1}\cdot u_ {q-1}'}$, hence we would have $\overline{[s_1,t_1]}_{q-1}=\overline{[t_1,s_1]}_{q-1}$, which would contradict the initial hypothesis.
So $i\ge q$.
Then $\lg_{S(X)}(\overline{[s_1,t_1]}_q)=q$, $u_q=(r_q,\dots,\widehat{r_i},\dots,r_\ell)$ is a reduced syllabic word of length $\ell-q$, and $\overline{w}=\overline{[s_1,t_1]_q\cdot u_q}$.
We prove in the same way that $\lg_{S(X)}(\overline{[t_1,s_1]}_q)=q$ and that there exists a reduced syllabic word $u_q'$ of length $\ell-q$ such that $\overline{w}=\overline{[t_1,s_1]_q\cdot u_q'}$.
This concludes the proof of the claim.

In the above claim $q$ is necessarily bounded by $q\le\ell$, hence there exist $q\ge 2$ and a reduced syllabic word $u$ of length $\ell-q$ such that $\overline{[s_1,t_1]}_q=\overline{[t_1,s_1]}_q$, $\lg_{S(X)}(\overline{[s_1,t_1]}_q)=q$ and $\overline{w}=\overline{[s_1,t_1]_q\cdot u}=\overline{[t_1,s_1]_q\cdot u}$.
As in the case $s_1=t_1$ treated at the beginning of the proof, we can go from $w$ to $[s_1,t_1]_q\cdot u$ through a finite sequence of elementary M-operations of type II, and we can go from $[t_1,s_1]_q\cdot u$ to $w'$ through a finite sequence of elementary M-operations of type II.
Obviously we can also go from $[s_1,t_1]_q \cdot u$ to $[t_1,s_1]_q \cdot u$ through a single elementary M-operation of type II.
So, we can go from $w$ to $w'$ through a finite sequence of elementary M-operations of type II.

It remains to show that, if a syllabic word $w=(x_1^{a_1},\dots,x_\ell^{a_\ell})$ is M-reduced, then $w$ is reduced.
We argue by induction on the length $\ell$ of the syllabic word.
The case $\ell=1$ is trivial, hence we can assume that $\ell\ge 2$ and that the induction hypothesis holds.

Suppose $w$ is not reduced.
By the induction hypothesis $w_1=(x_2^{a_2},\dots,x_\ell^{a_\ell})$ is reduced.
So, by Lemma \ref{lem3_2}, there exists $i\in\{2,\dots,\ell\}$ such that
\[
x_1=x_2^{a_2}\cdots x_{i-1}^{a_{i-1}}x_ix_{i-1}^{-a_{i-1}}\cdots x_2^{-a_2}\,.
\]
Let $w_2=(x_1^{a_1},x_2^{a_2},\dots,x_{i-1}^{a_{i-1}})$ and $w_2'=(x_2^{a_2},\dots,x_{i-1}^{a_{i-1}},x_i^{a_1})$.
By the above equality we have $\overline{w_2}=\overline{w_2'}$.
Moreover $w_2$ is M-reduced since $w$ is M-reduced, hence, by the induction hypothesis, $w_2$ is reduced.
Then $w_2'$ is also reduced since it has the same length as $w_2$.
Thus, by what is proved above, we can go from $w_2$ to $w_2'$ through a finite sequence of elementary M-operations of type II.
Set $w'=(x_2^{a_2},\dots,x_{i-1}^{a_{i-1}},x_{i+1}^{a_{i+1}},\dots,x_\ell^{a_\ell})$ if $a_1+a_i=0$ and $w'=(x_2^{a_2},\dots,x_{i-1}^{a_{i-1}},x_i^{a_1+a_i},x_{i+1}^{a_{i+1}},\dots,x_\ell^{a_\ell})$ if $a_1+a_i\neq0$.
Then we can go from $w$ to $w'$ through a finite sequence of elementary M-operations.
Since $\lg_{S(X)}(w')<\lg_{S(X)}(w)$, this contradicts the hypothesis that $w$ is M-reduced.
\end{proof}


\section{quasi-Dyer systems}\label{sec4}

Recall that the aim of this section is to prove Theorem \ref{thm2_3}.

\begin{lem}\label{lem4_1}
Let $(G,X)$ be a strongly marked group with Property $\DD$.
Let $x,y\in X$, $x\neq y$, $a\in\Z_{o(x)}\setminus\{0\}$ and $b\in\Z_{o(y)} \setminus\{0\}$.
If $x^ay^b=y^bx^a$, then $xy=yx$.
\end{lem}

\begin{proof}
Suppose $a\neq 1$.
Let $w=(x^{a-1},y^b,x^{-1})$ and $w'=(x^{-1},y^b,x^{a-1})$.
We have $\overline{w}=x^{a-1}y^bx^{-1}=x^{-1}y^{b}x^{a-1}=\overline{w'}$.
Thus, since $(G,X)$ has Property $\DD$, either we can reduce the length of both, $w$ and $w'$, by applying elementary M-operations, if $\lg_{S(X)}(\overline{w})=\lg_{S(X)}(\overline{w'})<3$, or we can go from $w$ to $w'$ through a finite sequence of elementary M-operations of type II, if $\lg_{S(X)}(\overline{w})=\lg_{S(X)}(\overline{w'})=3$.
In both cases we must be able to apply an elementary M-operation to $w$.
We cannot apply any elementary M-operation of type I to $w=(x^{a-1},y^b,x^{-1})$ since $x\neq y$ and $(G,X)$ is strongly marked.
Since $a-1\neq-1$, the only elementary M-operations that could be applied to $w$ are
\[
(x^{a-1},y^b,x^{-1})\to(y^b,x^{a-1},x^{-1})\quad\text{or}\quad(x^ {a-1},y^b,x^{-1})\to(x^{a-1},x^{-1},y^b)\,.
\]
So, either $x^{a-1}y^b=y^bx^{a-1}$ or $x^{-1}y^b=y^bx^{-1}$.
Any of these two equalities combined with the equality $x^ay^b=y^bx^a$ implies $xy^b=y^bx$.

If $b\neq1$, then we use the equality $y^{b-1}xy^{-1}=y^{-1}xy^{b-1}$ to show in the same way that $xy=yx$.
\end{proof}

\begin{lem}\label{lem4_2}
Let $(G,X)$ be a strongly marked group with Property $\DD$.
Let $x,y\in X$, $x\neq y$, $a\in\Z_{o(x)}\setminus\{0\}$, $b\in\Z_{o(y)}\setminus\{0\}$, and $m\ge3$.
Assume that $xy\neq yx$, $\overline{[x^a,y^b]}_m=\overline{[y^b,x^a]}_m$ and $\lg_{S(X)}(\overline{[x^a,y^b]}_m)=m$.
Then $o(x)$ and $o(y)$ are both finite and even, $a=o(x)/2$, and $b=o(y)/2$.
Moreover, $m$ is unique in the sense that, if $\overline{[x^a,y^b]}_n=\overline{[y^b,x^a]}_n$ and $\lg_{S (X)}(\overline{[x^a,y^b]}_n)=n$, then $n=m$.
\end{lem}

\begin{proof}
We assume that $a\in\Z_{o(x)}\setminus\{0\}$ and $b\in\Z_{o(y)}\setminus\{0\}$ are fixed and we choose $m$ minimal so that $\overline{[x^a,y^b]}_m=\overline{[y^b,x^a]}_m$ and $\lg_{S(X)}(\overline{[x^a,y^b]}_m)=m$.
We have $m\ge 3$ since $x$ and $y$ do not commute (see Lemma \ref{lem4_1}).
We start by showing that $o(x)$ and $o(y)$ are both finite and even, that $a=o(x)/2$, and that $b=o(y)/2$.
Suppose $m$ is even.
Let 
\[
w=(x^{-a})\cdot[y^{b},x^{a}]_{m-1}\quad\text{and}\quad w'=[y^{b},x^{a}]_{m-1}\cdot(x^{-a})\,.
\]
We have $\overline{w}=\overline{w'}$.
Since $(G,X)$ has Property $\DD$, either we can reduce the length of both, $w$ and $w'$, by applying elementary M-operations, if $\lg_{S(X)}(\overline{w})=\lg_{S(X)}(\overline{w'})<m$, or we can go from $w$ to $w'$ through a finite sequence of elementary M-operations of type II, if $\lg_{S(X)}(\overline{w})=\lg_{S(X)}(\overline{w'})=m$.
In both cases we must be able to apply an elementary M-operation to $w$.
We cannot apply any elementary M-operation of type I to $w$ because $x\neq y$.
If we had $-a\neq a$, then, by the minimality of $m$, we could not apply any elementary M-operation of type II to $w$ either.
So, to be able to apply an elementary M-operation to $w$ we must have $-a=a$.
This is possible only if $o(x)$ is finite and even and $a=o(x)/2$.
We show in the same way that $o(y)$ is finite and even and that $b=o(y)/2$.
The case where $m$ is odd is treated in the same way with the words $w=(x^{-a})\cdot[y^b,x^a]_{m-1}$ and $w'=[y^b,x^a]_{m-1}\cdot(y^{-b})$.

It remains to show that $m$ is unique.
Suppose there exists another integer $n\ge3$ different from $m$ such that $\overline{[x^a,y^b]}_n=\overline{[y^b,x^a]}_n$ and $\lg_{ S(X)}(\overline{[x^a,y^b]}_n)=n$.
Since $m$ was chosen minimal, we have $n>m$.
Assume $m$ is even.
Then we have the following sequence of elementary M-operations:
\[
[x^a,y^b]_n\to[y^b,x^a]_m\cdot[x^{a},y^b]_{n-m}\to[y^b,x^a] _{m-1}\cdot[y^b, x^a]_{n-m-1}\,,
\]
hence $[x^a,y^b]_n$ is not M-reduced.
This contradicts the hypothesis $\lg_{ S(X)}(\overline{[x^a,y^b]}_n)=n$.
The case where $m$ is odd is treated in the same way.
\end{proof}

\begin{lem}\label{lem4_3}
Let $(G,X)$ be a strongly marked group with Property $\DD$.
Let $x,y\in X$, $x\neq y$, $a\in\Z_{o(x)}\setminus\{0\}$, $b\in\Z_{o(y)}\setminus\{0\}$, and $m\ge 3$ even.
Assume $xy\neq yx$, $\overline{[x^a,y^b]}_m=\overline{[y^b,x^a]}_m$ and $\lg_{S(X)}(\overline{[x^a,y^b]}_m)=m$.
Then $o(x)=o(y)=2$ and $a=b=1$.
\end{lem}

\begin{proof}
We know by Lemma \ref{lem4_2} that $o(x)$ and $o(y)$ are both finite and even, that $a=o(x)/2$ and that $b=o(y)/2$.
We also know that $m$ is unique.
Suppose $o(x)\neq 2$.
We choose $c\in\Z_{o(x)}$ such that $c\neq0$ and $c\neq a$.
Let
\[
w=(x^{a-c})\cdot[y^b,x^a]_{m-1}\cdot(x^{-c})\quad\text{and}\quad w'=(x^{-c})\cdot[y^b,x^a]_{m-1}\cdot(x^{a-c})\,.
\]
Since $\overline{[x^a,y^b]}_m=\overline{[y^b,x^a]}_m$, we have $\overline{w}=\overline{w'}$.
So, either we can reduce the length of both, $w$ and $w'$, by applying elementary M-operations, if $\lg_{S(X)}(\overline{w})=\lg_{S(X)}(\overline{w'})<m+1$, or we can go from $w$ to $w'$ through a finite sequence of elementary M-operations of type II, if $\lg_{S(X)}(\overline{w})=\lg_{S(X)}(\overline{w'})=m+1$.
In both cases we must be able to apply an elementary M-operation to $w$.
We cannot apply any elementary M-operation of type I to $w$ since $x\neq y$.
We cannot apply any elementary M-operation of type II to $w$ either since $m$ is unique, $a-c\neq a$, and $-c\neq a$ (recall that $-a=a=o(x)/2$).
This is a contradiction, hence we necessarily have $o(x)=2$ and $a=1$.
We prove in the same way that $o(y)=2$ and $b=1$.
\end{proof}

\begin{lem}\label{lem4_4}
Let $(G,X)$ be a strongly marked group with Property $\DD$.
Let $x,y\in X$, $x\neq y$, $a\in\Z_{o(x)}\setminus\{0\}$, $b\in\Z_{o(y)}\setminus\{0\}$ and $m\ge3$ odd.
Assume $xy\neq yx$, $\overline{[x^a,y^b]}_m=\overline{[y^b,x^a]}_m$ and $\lg_{S(X)}(\overline{[x^a,y^b]}_m)=m$.
Then either $o(x)=2$ (and $a=1$), or $o(y)=2$ (and $b=1$).
\end{lem}

\begin{proof}
We know by Lemma \ref{lem4_2} that $o(x)$ and $o(y)$ are both finite and even, that $a=o(x)/2$ and that $b=o(y)/2$.
We also know that $m$ is unique.
Suppose $o(x)\neq2$ and $o(y)\neq2$.
We choose $c\in\Z_{o(x)}$ and $d\in\Z_{o(y)}$ such that $c\neq0$, $c\neq a$, $d\neq0$ and $d\neq b$.
Let 
\[
w=(x^{a-c})\cdot[y^b,x^a]_{m-1}\cdot(y^{-d})\quad\text{and}\quad w'=(x^{-c})\cdot[y^b,x^a]_{m-1}\cdot(y^{b-d})\,.
\]
Since $\overline{[x^a,y^b]}_m=\overline{[y^b,x^a]}_m$, we have $\overline{w}=\overline{w'}$.
So, either we can reduce the length of both, $w$ and $w'$, by applying elementary M-operations, if $\lg_{S(X)}(\overline{w})=\lg_{S(X)}(\overline{w'})<m+1$, or we can go from $w$ to $w'$ through a finite sequence of elementary M-operations of type II, if $\lg_{S(X)}(\overline{w})=\lg_{S(X)}(\overline{w'})=m+1$.
In both cases we must be able to apply an elementary M-operation to $w$.
We cannot apply any elementary M-operation of type I to $w$ since $x\neq y$.
We cannot apply any elementary M-operation of type II to $w$ either since $m$ is unique, $a-c\neq a$, and $-d\neq b$ (recall that $-b=b=o(y)/2$).
This is a contradiction, hence we necessarily have $o(x)=2$ or $o(y)=2$.
\end{proof}

\begin{proof}[Proof of Theorem \ref{thm2_3}]
Let $(G,X)$ be a strongly marked group with Property $\DD$.
We start by defining a simplicial graph $\Gamma$ and maps $m:E(\Gamma)\to\N_{\ge2}$ and $f:V(\Gamma)\to\N_{\ge2}\cup \{\infty\}$.
The set $V(\Gamma)$ is a set in one-to-one correspondence with $X$, and we set $X=\{x_v\mid v\in V(\Gamma)\}$.
We set $f(v)=o(x_v)$ for all $v\in V(\Gamma)$.
A pair $e=\{u,v\}$ belongs to $E(\Gamma)$ if and only if there exist $a\in\Z_{f(u)}\setminus\{0\}$, $b\in\Z_{f(v)}\setminus\{0\}$ and $m\ge2$ such that $\overline{[x_u^a,x_v^b]}_m=\overline{[x_v^b,x_u^a]}_m$ and $\lg_{S(X)}(\overline{[x_u^a,x_v^b]}_m)=m$.
We know by Lemma \ref{lem4_2} that such an $m$ is unique if it exists.
In that case we set $m(e)=m$.
It follows from Lemmas \ref{lem4_2}, \ref{lem4_3} and \ref{lem4_4} that the triple $(\Gamma,m,f)$ defines a quasi-Dyer group $QD=QD(\Gamma,m,f)$.

In order to differentiate the standard generators of $QD$ from the elements of $X$, we denote by $Y=\{y_v\mid v\in V(\Gamma)\}$ the standard generating set for $QD$.
By Lemmas \ref{lem4_1}, \ref{lem4_3} and \ref{lem4_4} we have a homomorphism $\varphi:QD\to G$ which sends $y_v$ to $x_v$ for all $v\in V(\Gamma)$.
This homomorphism is surjective since $X$ generates $G$.
On the other hand we define a set-map $\psi:G\to QD$ as follows.
Let $g\in G$.
We choose a syllabic expression $w=(x_{v_1}^{a_1},x_{v_2}^{a_2},\dots,x_{v_\ell}^{a_\ell})$ for $g$ and we set $\psi(g)=y_{v_1}^{a_1}y_{v_2}^{a_2}\cdots y_{v_\ell}^{a_\ell}\in QD$.
The fact that $(G,X)$ has Property $\DD$ combined with Lemmas \ref{lem4_1}, \ref{lem4_3} and \ref{lem4_4} implies that $\psi(g)$ is well-defined in the sense that its definition does not depend on the choice of the syllabic expression.
Now we show that $\psi\circ\varphi=\id_{QD}$.
This implies that $\varphi$ is also injective and therefore ends the proof of the theorem.

Let $h\in QD$.
Let $(y_{v_1}^{a_1},y_{v_2}^{a_2},\dots,y_{v_\ell}^{a_\ell})$ be a syllabic expression of $h$.
Then $(x_{v_1}^{a_1},x_{v_2}^{a_2},\dots,x_{v_\ell}^{a_\ell})$ is a syllabic expression of $\varphi(h)$, hence $(\psi\circ\varphi)(h)=y_{v_1}^{a_1}y_{v_2}^{a_2}\cdots y_{v_\ell}^{a_\ell}=h$.
So, $\psi\circ\varphi=\id_{QD}$.
\end{proof}


\section{Two generators quasi-Dyer groups}\label{sec5}

In this section we study the group
\[
QD_{m,k}=\langle x,y\mid x^2=y^{2k}=1\,,\ [x,y^k]_m=[y^k,x]_m\rangle\,,
\]
where $m\ge3$ is odd and $k\ge2$.
We start by showing that $(QD_{m,k},\{x,y\})$ is strongly marked.

\begin{lem}\label{lem5_1}
Let $m\ge3$ odd and $k\ge2$.
Then $(QD_{m,k},\{x,y\})$ is strongly marked.
\end{lem}

\begin{proof}
We have $X=\{x,y\}$ and $S(X)=\{x,y,y^2,\dots,y^{2k-1}\}$.
We choose $b\in\{1,\dots,2k-1\}$ and we show that $xy^b\not\in S(X)\cup\{1\}$.
Let 
\[
D_m=\langle x,y'\mid x^2=y'^2=1\,,\ [x,y']_m=[y',x]_m \rangle\,,\
C_{2k}=\langle y\mid y^{2k}=1\rangle\,,\
K=\langle z\mid z^2=1\rangle\,.
\]
Observe that $D_m$ is a dihedral group of order $2m$, $C_{2k}$ is a cyclic group of order $2k$, and $K$ is a cyclic group of order $2$.
We have an embedding of $K$ into $D_m$ which sends $z$ to $y'$, we have an embedding of $K$ into $C_{2k}$ which sends $z$ to $y^k$, and we have $QD_{m,k} = D_m*_KC_{2k}$.
If $b\neq k$, then $y^b\in C_{2k}\setminus K$ and $x\in D_m\setminus K$, hence, by the general theory of normal forms in amalgamated products of groups (see Serre \cite{Serre1}), $xy^b\not\in D_m$ and $xy^b\not\in C_{2k}$, hence $xy^b\not\in S(X)\cup\{1\}$.
Suppose $b=k$.
Then $y^b=y'\in D_m$, hence $xy^b=xy'$ is a non-trivial rotation, thus $xy^b\not\in\{x,y',1\}=D_m\cap(S(X)\cup\{1\})$, and therefore $xy^b\not\in S(X)\cup\{1\}$.
\end{proof}

The rest of the proof of Proposition \ref{prop2_4} uses rewriting systems.
Since we do not assume the reader to be familiar with them, we start by giving the necessary background for understanding our proof, and we refer to Cohen \cite{Cohen1} and Le Chenadec \cite{LeChe1} for detailed explanations.

Let $S$ be a finite set (called an \emph{alphabet}) and let $S^*$ be the free monoid on $S$.
The elements of $S^*$ are called \emph{words} and they are written as finite sequences, as for syllabic words.
The empty word is denoted by $\epsilon$.
The concatenation of two words $w,w'\in S^*$ is written $w\cdot w'$.
A \emph{rewriting system} on $S^*$ is a subset $R\subset S^*\times S^*$.
Let $w,w'\in S^*$.
We set $w\to_R w'$ or simply $w\to w'$ if there exist $w_1,w_2\in S^*$ and $(u,v)\in R$ such that $w=w_1\cdot u\cdot w_2$ and $w'=w_1\cdot v\cdot w_2$.
More generally, we set $w\to_R^*w'$ or simply $w\to^*w'$ if $w'=w$ or if there exists a finite sequence $w_0=w,w_1,\dots,w_p=w'$ in $S^*$ such that $w_{i-1}\to w_i$ for all $i\in\{1,\dots,p\}$.
A word $w\in S^*$ is \emph{$R$-reducible} if there exists a word $w'\in S^*$ such that $w\to w'$.
We say that $w$ is \emph{$R$-irreducible} otherwise.
The pair $(S,R)$ is a \emph{rewriting system for a monoid} $M$ if $\langle S\mid u=v\text{ for }(u,v)\in R\rangle^+$ is a monoid presentation for $M$.
A rewriting system for a group $G$ is a rewriting system for $G$ viewed as a monoid.
In particular, in that case $S$ must generate $G$ as a monoid.
If $(S,R)$ is a rewriting system for a monoid $M$ and $w=(s_1,\dots,s_\ell)\in S^*$, then we denote by $\overline{w}= s_1s_2\cdots s_\ell$ the element of $M$ represented by $w$.

Let $R$ be a rewriting system.
We say that $R$ is \emph{Noetherian} if there is no infinite sequence $w_0\to w_1\to w_2\to\cdots$ in $S^*$.
On the other hand we say that $R$ is \emph{confluent} if, for all $u,v_1,v_2\in S^*$ such that $u\to^*v_1$ and $u\to^*v_2$, there exists $w\in S^*$ such that $v_1\to^*w$ and $v_2\to^*w$.
We say that $R$ is \emph{complete} if it is both, Noetherian and confluent.

Theorems \ref{thm5_2} and \ref{thm5_3} below contain classical results on rewriting systems that we will use in our proof of Proposition \ref{prop2_4}.

\begin{thm}[Newman \cite{Newma1}]\label{thm5_2}
Let $(S,R)$ be a complete rewriting system for a monoid $M$.
\begin{itemize}
\item[(1)]
For each $w'\in S^*$ there exists a unique $R$-irreducible word $w\in S^*$ such that $w'\to^* w$.
\item[(2)]
For each $g\in M$ there exists a unique $R$-irreducible word $w\in S^*$ such that $\overline{w}=g$.
\end{itemize}
\end{thm}

Suppose $S^*$ is endowed with a total order $\le$ such that:
\begin{itemize}
\item[(a)]
there is no infinite descending chain $w_0>w_1>w_2>\cdots$ in $S^*$, and
\item[(b)]
for all $u,v,w_1,w_2\in S^*$, if $u>v$, then $w_1\cdot u\cdot w_2>w_1\cdot v\cdot w_2$.
\end{itemize}
For instance, if $S$ itself is endowed with a total order, then the graded lexicographical order on $S^*$, denoted $\le_{\Glex}$, is a total order on $S^*$ satisfying Conditions (a) and (b).
Then a rewriting system $R$ on $S^*$ satisfying $u>v$ for all $(u,v)\in R$ is Noetherian.

A \emph{critical pair} in a rewriting system $R$ is a quintuple $(u_1,u_2,u_3,v_1,v_2)$ of elements of $S^*$ satisfying one of the following two conditions:
\begin{itemize}
\item[(a)]
$(u_1\cdot u_2,v_1)\in R$, $(u_2\cdot u_3,v_2)\in R$, and $u_2\neq\epsilon$,
\item[(b)]
$(u_1\cdot u_2\cdot u_3,v_1)\in R$ and $(u_2,v_2)\in R$.
\end{itemize}
We say that a critical pair $(u_1,u_2,u_3,v_1,v_2)$ is \emph{resolved} if there exists $w\in S^*$ such that
\begin{itemize}
\item
$v_1\cdot u_3\to^*w$ and $u_1\cdot v_2\to^*w$ in Case (a),
\item
$v_1\to^*w$ and $u_1\cdot v_2\cdot u_3\to^*w$ in Case (b).
\end{itemize}

\begin{thm}[Newman \cite{Newma1}]\label{thm5_3}
Let $R$ be a Noetherian rewriting system.
If all critical pairs of $R$ are resolved, then $R$ is confluent.
\end{thm}

\begin{proof}[Proof of Proposition \ref{prop2_4}]
Let $m\ge3$ odd and $k\ge 2$, and let
\[
QD_{m,k}=\langle x,y\mid x^2=y^{2k}=1\,,\ [x,y^k]_m=[y^k,x]_m\rangle\,.
\]
We already know by Lemma \ref{lem5_1} that $(QD_{m,k},\{x,y\})$ is strongly marked.
So, it remains to show that $(QD_{m,k},\{x,y\})$ has Property $\DD$.

We set  $S=S(\{x,y\})=\{x,y,y^2,\dots,y^{2k-1}\}$ that we totally order by $x>y>y^2>\cdots>y^{2k-1}$, and we endow $S^*$ with the graded lexicographic order $\le_\Glex$.
Let $R$ be the following rewriting system on $S^*$:
\begin{gather*}
R=\{(x,x)\to\epsilon,[x,y^k]_m\to[y^k,x]_m\}\,\cup\ \{(y^a,y^{-a})\to \epsilon\mid a\in\Z_{2k}\setminus\{0\}\}\\
\cup\ \{(y^a,y^b)\to(y^{a+b})\mid a,b\in\Z_{2k}\setminus \{0\}\,,\ a+b\neq 0\}\,.
\end{gather*}
It is clear that $(S,R)$ is a rewriting system for $QD_{m,k}$ viewed as a monoid.
It is also clear that $R$ is Noetherian since $u>_{\Glex}v$ for all $(u,v)\in R$.
Another immediate observation is that, if $w,w'\in S^*$ are such that $w\to^*w'$, then we can go from $w$ to $w'$ through a finite sequence of elementary M-operations.
Now we prove that $R$ is confluent.

We list below all critical pairs $(u_1,u_2,u_3,v_1,v_2)$ of type (a), and for each of them we give a word $w\in S^*$ such that $u_1\cdot v_2\to^*w$ and $v_1\cdot u_3\to^*w$.
There is no (non-trivial) critical pair of type (b), hence, by Theorem \ref{thm5_3}, these calculations show that $R$ is confluent.
\begin{itemize}
\item
$(u_1,u_2,u_3,v_1,v_2)=((x),(x),(x),\epsilon,\epsilon)$ and $w=(x)$.
\item
$(u_1,u_2,u_3,v_1,v_2)=((x),(x),[y^k,x]_{m-1},\epsilon,[y^k,x]_m)$ and $w=[y^k,x]_{m-1}$.
\item
$(u_1,u_2,u_3,v_1,v_2)=([x,y^k]_{m-1},(x),(x),[y^k,x]_m,\epsilon)$ and $w=[x,y^k]_{m-1}$.
\item
$(u_1,u_2,u_3,v_1,v_2)=([x,y^k]_{2a},[x,y^k]_{2b-1},[y^k,x]_{2a},[y^k,x]_m,[y^k,x]_m)$ and $w=[y^k,x]_{2b-1}$, where $a,b\in\N_{\ge 1}$ and $2a+2b-1=m$.
\item
$(u_1,u_2,u_3,v_1,v_2)=((y^a),(y^{-a}),(y^a),\epsilon,\epsilon)$ and $w=(y^a)$, where $a\in\Z_{2k}\setminus \{0\}$.
\item
$(u_1,u_2,u_3,v_1,v_2)=((y^a),(y^{-a}),(y^b),\epsilon,(y^{-a+b}))$ and $w=(y^b)$, where $a,b\in\Z_{2k}\setminus\{0\}$ and $-a+b\neq0$.
\item
$(u_1,u_2,u_3,v_1,v_2)=((y^a),(y^b),(y^{-b}),(y^{a+b}),\epsilon)$ and $w=(y^a)$, where $a,b\in\Z_{2k}\setminus\{0\}$ and $a+b\neq0$.
\item
$(u_1,u_2,u_3,v_1,v_2)=((y^a),(y^b),(y^c),(y^{a+b}),(y^{b+c}))$, $w=\epsilon$ if $a+b+c=0$, and $w=(y^{a+b+c})$ if $a+b+c\neq0$, where $ a,b,c\in\Z_{2k}\setminus\{0\}$, $a+b\neq0$ and $b+c\neq0$.
\end{itemize}

In order to prove that $QD_{m,k}$ has Property $\DD$ it suffices to show that, if $w$ and $w'$ are two M-reduced syllabic words representing the same element of $QD_{m, k}$, then we can go from $w$ to $w'$ through a finite sequence of elementary M-operations of type II.
Since $R$ is confluent, by Theorem \ref{thm5_2} there exists a unique $R$-irreducible word $w_0\in S^*$ such that $w\to^*w_0$ and $w'\to^*w_0$.
As previously indicated, this implies that we can go from $w$ to $w_0$ through a finite sequence of elementary M-operations.
Moreover, $\lg(w)\ge\lg(w_0)$ (since $w\ge_{\Glex}w_0$) and $w$ is M-reduced, hence all these elementary M-operations must be of type II.
Similarly, we can go from $w'$ to $w_0$ through a finite sequence of elementary M-operations of type II.
So, we can go from $w$ to $w'$ through a finite sequence of elementary M-operations of type II.
\end{proof}


\section{Parabolic subgroups}\label{sec6}

Let $(G,X)$ be a strongly marked group with Property $\DD$ and let $QD(\Gamma,m,f)$ be the quasi-Dyer presentation for $(G,X)$.
Recall that, for $U\subset V(\Gamma)$, we denote by $\Gamma_U$ the full subgraph of $\Gamma$ spanned by $U$, we denote by $m_U:E(\Gamma_U)\to\N_{\ge2}$ the restriction of $m$ to $E(\Gamma_U)$, and we denote by $f_U:V(\Gamma_U)\to\N_{\ge 2}\cup\{\infty\}$ the restriction of $f$ to $V(\Gamma_U)$.
We begin by proving Proposition \ref{prop2_7}.

\begin{proof}[Proof of Proposition \ref{prop2_7}]
Let $U\subset V(\Gamma)$.
Let $Y=\{x_u\mid u\in U\}$.
From the inclusion $S(Y)\subset S(X)$ it follows that $(G_Y,Y)$ is strongly marked.
Again, to show that $(G_Y,Y)$ has Property $\DD$, it suffices to show that, if $w$ and $w'$ are two M-reduced syllabic words in $S(Y)^*$ representing the same element of $G_Y$, then we can go from $w$ to $w'$ through a finite sequence of elementary M-operations of type II.
Let $w$ and $w'$ be two M-reduced syllabic words in $S(Y)^*$ representing the same element of $G_Y$.
Since the elementary M-operations preserve $S(Y)^*$, $w$ and $w'$ are M-reduced seen as elements of $S(X)^*$.
Since $(G,X)$ has Property $\DD$, it follows that we can go from $w$ to $w'$ through a finite sequence of elementary M-operations of type II, and these elementary M-operations preserve $S(Y)^*$.
So, $(G_Y,Y)$ has Property $\DD$.
Finally, the fact that $QD(\Gamma_U,m_U,f_U)$ is the quasi-Dyer presentation for $(G_Y,Y)$ follows from the definition of a quasi-Dyer presentation.
\end{proof}

As indicated in Section \ref{sec2}, the rest of the section concerns only Dyer systems.
We start with the proof of Proposition \ref{prop2_8}.

\begin{proof}[Proof of Proposition \ref{prop2_8}]
Let $(D,X)$ be a Dyer system.
We first prove Part (2).
Part (1) will follow from Part (2). 
Let $Y\subset X$ and let $g\in G$.
Let $g_0\in D_Yg$ be of minimal syllabic length in $D_Yg$.
To prove Part (2) it suffices to show that $\lg_{S(X)}(hg_0)=\lg_{S(X)}(h)+\lg_{S(X)}(g_0) $ for all $h\in D_Y$.

We use the definitions and the notations of Section \ref{sec3}.
Let $(x_1^{a_1},\dots,x_p^{a_p})$ be a reduced syllabic expression for $g_0$ and let $(y_1^{b_1},\dots,y_q^{b_q})$ be a reduced syllabic expression for $h$.
For $i\in\{1,\dots,p\}$ and $j\in\{1,\dots,q\}$ we set
\[
\rho_i=x_1^{a_1}\cdots x_{i-1}^{a_{i-1}}x_ix_{i-1}^{-a_{i-1}}\cdots x_1^{-a_1}\quad\text{and}\quad \sigma_j=y_1^{b_1}
\cdots y_{j-1}^{b_{j-1}}y_jy_{j-1}^{-b_{j-1}}\cdots y_1^{-b_1}\,.
\]
We have
\[
N(g_0)=\sum_{i=1}^pa_i[\rho_i]\quad\text{and}\quad N(h)=\sum_{j=1}^qb_j[\sigma_j]\,.
\]
By Theorem \ref{thm3_1} the $\rho_i$ are pairwise distinct and the $\sigma_j$ are pairwise distinct.

Set $R_Y=\{kyk^{-1}\mid k\in D_Y\text{ and }y\in Y\}$.
Since elementary M-operations preserve $S(Y)^*$ it follows from Theorem \ref{thm2_2} that $y_j\in Y$ for all $j\in\{1,\dots,q\}$, hence $\sigma_j\in R_Y$ for all $j\in\{1,\dots ,q\}$.
Suppose there exists $i\in\{1,\dots,p\}$ such that $\rho_i\in R_Y$.
Let $g_0'=x_1^{a_1}\cdots x_{i-1}^{a_{i-1}}x_{i+1}^{a_{i+1}}\cdots x_p^{a_p}$.
Then $\lg_{S(X)}(g_0')<\lg_{S(X)}(g_0)=p$, $g_0=\rho_i^{a_i}g_0'$ and $\rho_i^{a_i}\in D_Y$, which would contradict the minimalilty of the syllabic length of $g_0$ in $D_Yg$.
So, $\rho_i\not\in R_Y$ for all $i\in\{1,\dots,p\}$.
Since $h\in D_Y$, it follows that $h\rho_ih^{-1}\not\in R_Y$ for all $i\in\{1,\dots,p\}$.

By Theorem \ref{thm3_1},
\[
N(hg_0)=N(h)+h\cdot N(g_0)=\sum_{j=1}^q b_j[\sigma_j]+\sum_{i=1}^pa_i[h\rho_i h^{-1}]\,.
\]
We know that the $\sigma_j$ are pairwise distinct and we know that the $h\rho_ih^{-1}$ are pairwise distinct.
We also know that $\sigma_j\neq h\rho_ih^{-1}$ for $j\in\{1,\dots,q\}$ and $i\in\{1,\dots,p\}$, since $\sigma_j\in R_Y$ and $h\rho_ih^{-1}\not\in R_Y$.
By Theorem \ref{thm3_1} it follows that 
\[
\lg_{S(X)}(hg_0)=q+p=\lg_{S(X)}(h)+\lg_{S(X)}(g_0)\,.
\]

Now we prove Part (1).
Let $g_0$ be an element of minimal syllabic length in $gD_Y$.
Then $g_0^{-1}$ is of minimal syllabic length in $D_Yg^{-1}$.
By Part (2) already proved, it follows that, for all $h\in D_Y$,
\[
\lg_{S(X)}(g_0h)=\lg_{S(X)}(h^{-1}g_0^{-1})=\lg_{S(X)}(h^{-1})+\lg_{S(X)}(g_0^{-1})=\lg_{S(X)}(g_0)+
\lg_{S(X)}(h)\,.
\proved
\]
\end{proof}

The rest of the section is dedicated to the proof of Theorem \ref{thm2_10}.
Lemmas \ref{lem6_1}, \ref{lem6_2} and \ref{lem6_3} are preliminaries to its proof, but they are also interesting by themselves.
Note that Lemmas \ref{lem6_1} and \ref{lem6_2} hold for all Dyer groups, while Lemma \ref{lem6_3} requires the Dyer group to be of finite type.

\begin{lem}\label{lem6_1}
Let $(D,X)$ be a Dyer system, let $Y,Y'\subset X$, and let $g_0 \in D$.
Assume $g_0$ is of minimal syllabic length in the double coset $D_{Y}g_0D_{Y'}$.
Set $Z=Y\cap(g_0Y'g_0^{-1})$.
Then $D_{Y}\cap(g_0D_{Y'}g_0^{-1})=D_Z$.
\end{lem}

\begin{proof}
The inclusion $D_Z\subset D_{Y}\cap(g_0D_{Y'}g_0^{-1})$ is obvious.
So, we only need to prove the reverse inclusion.
Let $g\in D_{Y}\cap(g_0D_{Y'}g_0^{-1})$.
Let $g'\in D_{Y'}$ such that $g=g_0g'g_0^{-1}$.
Let $(x_1^{a_1},\dots,x_p^{a_p})$ be a reduced syllabic expression for $g$, let $(y_1^{b_1},\dots,y_q^{b_q})$ be a reduced syllabic expression for $g'$, and let $(z_1^{c_1},\dots,z_\ell^{c_\ell})$ be a reduced syllabic expression for $g_0$.
Since elementary M-operations preserve $S(Y)^*$ and $S(Y')^*$, by Theorem \ref{thm2_2} we have $x_i\in Y$ for all $i\in\{1,\dots,p\}$ and $y_i\in Y'$ for all $i\in\{1,\dots, q\}$.
Moreover, $g_0g'=gg_0$ and $g_0$ is of minimal syllabic length in both, $D_{Y}g_0$ and $g_0D_{Y'}$, hence, by Proposition \ref{prop2_8},
\[
p+\lg_{S(X)}(g_0)=\lg_{S(X)}(gg_0)=\lg_{S(X)}(g_0g')=\lg_{S(X)}(g_0)+ q\,,
\]
thus $p=q$.
Now we show by induction on $p$ that $x_i\in Z$ for all $i\in\{1,\dots,p\}$. 
This implies that $g\in D_Z$.

Assume $p=q=1$.
So, $g=x_1^{a_1}$ and $g'=y_1^{b_1}$.
Let $w=(z_1^{c_1},\dots,z_\ell^{c_\ell}, y_1^{b_1})$.
The word $w$ is a syllabic expression for $g_0g'$, which is of syllabic length $\ell+1$, hence $w$ is reduced.
Moreover, $\overline{x_1^{-a_1}\cdot w}=g^{-1}g_0g'=g_0$ is of syllabic length $\ell$, hence, by Lemma \ref{lem3_2}, either
\begin{itemize}
\item[(i)]
there exists $k\in\{1,\dots,\ell\}$ such that $x_1=z_1^{c_1}\cdots z_{k-1}^{c_{k-1}}z_kz_{k-1}^{-c_{k-1}}\cdots z_1^{-c_1}$, or
\item[(ii)]
$x_1=g_0y_1g_0^{-1}\in Y\cap g_0Y'g_0^{-1}= Z$.
\end{itemize}
We cannot have $x_1=z_1^{c_1}\cdots z_{k-1}^{c_{k-1}}z_k z_{k-1}^{-c_{k-1}}\cdots z_1^{-c_1}$ with $k\in\{1,\dots,\ell\}$, otherwise $(x_1^{a_1},z_1^{c_1},\dots,z_\ell^{c_\ell})$ would not be a reduced syllabic word. 
So, $x_1\in Z$.

We assume that $p\ge2$ and that the induction hypothesis holds.
Let $w=(z_1^{c_1},\dots,z_\ell^{c_\ell},y_1^{b_1},\dots,y_p^{b_p})$.
The word $w$ is a syllabic expression for $g_0g'$, which is of syllabic length $\ell+p$, hence $w$ is reduced.
On the other hand, $\overline{x_1^{-a_1}\cdot w}=x_2^{a_2}\cdots x_p^{a_p}g_0$ is of syllabic length $\ell+p-1$.
So, by lemma \ref{lem3_2}, either 
\begin{itemize}
\item[(i)]
there exists $k\in\{1,\dots,\ell\}$ such that $x_1=z_1^{c_1}\cdots z_{k-1}^{c_{k-1}}z_kz_{k-1}^{-c_{k-1}}\cdots z_1^{-c_1}$, or
\item[(ii)]
there exists $j\in\{1,\dots,p\}$ such that $x_1=g_0y_1^{b_1}\cdots y_{j-1}^{b_{j-1}}y_jy_{j-1}^{-b_{j-1}}\cdots y_1^{-b_1}g_0^{-1}$.
\end{itemize}
Suppose there exists $k\in\{1,\dots,\ell\}$ such that $x_1=z_1^{c_1}\cdots z_{k-1}^{c_{k-1}}z_k z_{k-1}^{-c_{k-1}}\cdots z_1^{-c_1}$.
Let $g_0'=z_1^{c_1}\cdots z_{k-1}^{c_{k-1}}
\allowbreak
z_{k+1}^{c_{k+1}}\cdots z_\ell^{c_\ell}$.
Then $x_1^{c_k}\in D_Y$ and $g_0=x_1^{c_k}g_0'$.
This contradicts the minimality of the syllabic length of $g_0$ in $D_Yg_0D_{Y'}$.
So, there exists $j\in\{1,\dots,p\}$ such that $x_1=g_0y_1^{b_1}\cdots y_{j-1}^{b_{j-1}}y_jy_{j-1}^{-b_{j-1}}\cdots y_1^{-b_1}g_0^{-1}$.

Since $\lg_{S(X)}(x_1^{-a_1}gg_0)=\ell+p-1$, by Lemma \ref{lem3_2} we also have $b_j=a_1$ and
\[
(z_1^{c_1},\dots,z_\ell^{c_\ell},y_1^{b_1},\dots,y_{j-1}^{b_{j-1}},y_{j+1}^{b_{j+1}},\dots,y_p^{b_p})
\]
is a reduced syllabic expression for $x_1^{-a_1}gg_0$.
Let $h=x_2^{a_2}\cdots x_p^{a_p}$ and $h'=y_1^{b_1}\cdots y_{j-1}^{b_{j-1}}y_{j+1}^{b_{j+1}}\cdots y_p^{b_p}$.
Then $h\in D_Y$ (since $x_2,\dots,x_p\in Y$), $h'\in D_{Y'}$ (since $y_1,\dots,y_{j-1},y_{j+1},\dots,y_p\in Y'$), $(x_2^{a_2},\dots,x_p^{a_p})$ is a reduced syllabic expression for $h$, $(y_1^{b_1},\dots,y_{j-1}^{b_{j-1}},y_{j+1}^{b_{j+1}},\dots,y_p^{b_p})$ is a reduced syllabic expression for $h'$, and $h=g_0h'g_0^{-1}$.
By the induction hypothesis it follows that $x_i\in Z$ for all $i\in\{2,\dots,p\}$.
Finally, since $g\in D_Y\cap g_0D_{Y'}g_0^{-1}$ and $x_i\in D_Y\cap g_0D_{Y'}g_0^{-1}$ for all $i\in\{2,\dots,p\}$, we have $x_1^{a_1}\in D_Y\cap g_0D_{Y'}g_0^{-1}$, hence, by the case $p=1$ previously proved, $x_1\in Z$.
\end{proof}

\begin{lem}\label{lem6_2}
Let $(D,X)$ be a Dyer system.
Then a finite intersection of parabolic subgroups of $D$ is a parabolic subgroup of $D$.
\end{lem}

\begin{proof}
It suffices to show that, if $Y,Y'\subset X$ and $g\in D$, then $D_Y\cap gD_{Y'}g^{-1}$ is a parabolic subgroup.
Let $g_0$ be an element in $D_YgD_{Y'}$ of minimal syllabic length.
Let $Z=Y\cap g_0Y'g_0^{-1}$.
By Lemma \ref{lem6_1} we have $D_Y\cap g_0D_{Y'}g_0^{-1}=D_Z$.
Let $h\in D_Y$ and $h'\in D_{Y'}$ such that $g=hg_0h'$.
Then
\[
D_Y\cap gD_{Y'}g^{-1}=(hD_Yh^{-1})\cap(hg_0h'D_{Y'}h'^{-1}g_0^{-1}h^{-1})=h(D_Y\cap g_0D_{Y'}g_0^{-1})
h^{-1}=hD_Zh^{-1}\,.
\proved
\]
\end{proof}

\begin{lem}\label{lem6_3}
 Let $(D,X)$ be a Dyer system of finite type.
Let $Y,Y'\subset X$ and $g,h\in D$ such that $gD_{Y'}g^{-1}\subsetneq hD_{Y}h^{-1}$.
Then $|Y'|<|Y|$.
\end{lem}

\begin{proof}
Without loss of generality we can assume that $h=1$, else we conjugate both sides with $h^{-1}$ and we  replace $g$ by $h^{-1}g$.
Let $g_0$ be an element in $D_YgD_{Y'}$ of minimal syllabic length.
Let $k\in D_Y$ and $k'\in D_{Y'}$ such that $g=kg_0k'$.
Then
\[
gD_{Y'}g^{-1}\subsetneq D_{Y}\ \Leftrightarrow\ kg_0k'D_{Y'}k'^{-1}g_0^{-1}k^{-1}\subsetneq kD_{Y}k^{-1}\
\Leftrightarrow\ g_0D_{Y'}g_0^{-1}\subsetneq D_Y\,.
\]
So, we can assume that $g_0D_{Y'}g_0^{-1}\subsetneq D_Y$.
Let $y\in Y'$.
There exists $f_y\in D_Y$ such that $g_0yg_0^{-1}=f_y$, that is, $g_0y=f_yg_0$.
Since $g_0$ has minimal syllabic length in both, $g_0D_{Y'}$ and $D_Yg_0$, by Proposition \ref{prop2_8} we have
\[
1+\lg_{S(X)}(g_0)=\lg_{S(X)}(g_0y)=\lg_{S(X)}(f_yg_0)=\lg_{S(X)}(f_y)+ \lg_{S(X)}(g_0)\,,
\]
hence $\lg_{S(X)}(f_y)=1$, thus there exist $x\in Y$ and $a\in\Z_{o(x)}\setminus\{0\}$ such that $f_y=x^a$.
Using the same arguments as in the proof of Lemma \ref{lem6_1}, from the equalities $x^ag_0=g_0y$ and $\lg_{S(X)}(x^ag_0)=\lg_{S(X)}(g_0y)=\lg_{S(X)}(g_0)+1$ it follows that $a=1$ and $x=g_0yg_0^{-1}$.
So, $g_0Y'g_0^{-1}\subset Y$.
Since conjugation by $g_0$ is a bijection, this inclusion implies that $|Y'|\le|Y|$.
Moreover, if $|Y'|=|Y|$, then $g_0Y'g_0^{-1}=Y$, hence $g_0D_{Y'}g_0^{-1}=D_Y$, which is a contradiction.
So, $|Y'|<|Y|$.
\end{proof}

\begin{proof}[Proof of Theorem \ref{thm2_10}]
Let $(D,X)$ be a Dyer system of finite type.
Let $\{P_i\mid i\in I\}$ be a non-empty collection of parabolic subgroups of $(D,X)$.
Let $\PP$ be the set of all finite intersections of elements of $\{P_i\mid i\in I\}$.
We know by Lemma \ref{lem6_2} that the elements of $\PP$ are all parabolic subgroups.
We choose $P_0=g_0D_{Y_0}g_0^{-1}$ in $\PP$ with minimal $|Y_0|$ and we show that $\bigcap_{i\in I}P_i=P_0$.
Clearly, it suffices to show that $P_0\subset P$ for all $P\in\PP$.
Let $P=gD_Yg^{-1}\in\PP$ and $P'=P_0\cap P$.
We have $P'\in\PP$ by definition and by Lemma \ref{lem6_2} there exist $h\in D$ and $Z\subset X$ such that $P'=hD_Zh^{-1}$.
We have $|Z|\ge|Y_0|$ by minimality of $|Y_0|$ and $P'=hD_Zh^{-1}\subset P_0=g_0D_{Y_0}g_0^{-1}$, hence, by Lemma \ref{lem6_3}, $P'=P_0$.
So, $P_0=P'=P_0\cap P\subset P$.
\end{proof}



\end{document}